\documentclass{article}
\usepackage[utf8]{inputenc}
\usepackage[left = 1 in, right = 1 in, top = 1 in, bottom = 1 in]{geometry}
\usepackage{amsmath}
\usepackage{amsfonts}
\usepackage{mathtools}
\usepackage{tikz}
\usepackage{relsize}
\usepackage{amsthm}
\usepackage{graphicx}
\graphicspath{{./probability/}}
\usepackage{pythontex}

\newtheorem{theorem}{Theorem}
\newtheorem{lemma}{Lemma}
\newtheorem{openproblem}{Open Problem}

\title{Enumeration of Random Walk Positions in $L_1$-norm ball in $\mathbb{Z}^d$}
\author{Luchen Shi, Will McCance, Hongjie Zeng}
\date{September 2022}

\begin{document}

\maketitle

\begin{abstract}
    In this paper, we mainly concerned about deriving the general formula to count the possible positions of $n$ step random walk in $\mathbb{Z}^d$ with unit length in each step, which we denoted as $|P_n^{d}|$. For our results, we firstly propose a recurrence relation of the counting formula: $|P_n^{d+1}| = |P_n^d| + 2\sum_{k=0}^{n-1} |P_k^d|$. Next, we propose two methods in deriving the explicit formula of $|P_n^{d}|$ using generating functions and Faulhaber's formula. Finally, we reached our main theorem in the matrix representation of our formula. 
\end{abstract}

\section{Preliminaries and Notations}
In the Euclidean space $\mathbb{R}^d$, we define the $L_p$-norm for point $x=(x_1,x_2\cdots,x_d)$ to be
\begin{equation*}
    ||x||_{p}=(|x_1|^p+|x_2|^p+\cdots+|x_d|^p)^{\frac{1}{p}}
\end{equation*}
and an open ball around the origin with radius $r$ is given by the set
\begin{equation*}
    B^p(r)=\{x\in\mathbb{R}^n: ||x||_{p}=(|x_1|^p+|x_2|^p+\cdots+|x_d|^p)^{\frac{1}{p}} \le r\}
\end{equation*}
When $d=2$, "balls" corresponding to $L_1$-norm (often called the taxicab or Manhattan metric) are bounded by squares with their diagonals parallel to the coordinate axes; When $d=3$, "balls" are octahedra with axes-aligned body diagonals corresponding to $L_1$-norm. We will use $\mathbb{Z}^d$ to denote the $d$-dimensional integer grid, i.e., the set of $d$-tuples $x=(x^1,\cdot\cdot\cdot,x^d)$ where $x^i$ are all integers. Each $x$ in $\mathbb{Z}^d$ has $2d$ "nearest neighbors", points that are at a distance one away from it. To do simple random walk in $\mathbb{Z}^d$, the walker starts at the origin and at each integer time $n$ moves to one of the nearest neighbors, each with probability $\frac{1}{2d}$. Let $S_n=(S^1_{n},\cdot\cdot\cdot,S^d_{n})$ be the position of the walker after $n$ steps, which is inside the $L_1$-norm ball $B^1_{n}$ with radius $n$. Then $$S_n=X_1+\cdots+X_n$$
where $X_i=(X^1_{i},\cdots,X^d_{i})$ and the $X_i$ are independent random vectors with $\mathbb{P}\{X_i=y\}=(2d)^{-1}$ for each $y\in\mathbb{Z}^d $ that is distance one from the origin. 
Define $W_n^d$ to be the set containing the sequences of steps in all possible random walks with $n$ steps:\\
\begin{equation*}
W_n^d  = \{X_i=(X^1_{i},\cdots,X^d_{i})| \, X_i \in \{+1,-1\}^d\,\,\,(i=1,\ldots,n)\}
\end{equation*}
We have $|W_n^d|=(2d)^n$, since each of the $n$ steps has $2d$ possibilities and are independent. Let us call a random walk of $n$ steps an “$n$-walk." From $W_n^d$, we can construct the set $P_n^d$ containing all possible positions in $\mathbb{Z}^d$ after an $n$-walk by summing each unit-length step:\\
\begin{equation*} 
P_n^d = \mathlarger{\mathlarger{\{}}\sum_{i=1}^n X_i | X_i=(X^1_{i},\cdots,X^d_{i}), X^k_{i}\in \{+1,-1\} \mathlarger{\mathlarger{\}}}
\end{equation*}
After these preparations, we are ready to derive the recursive formula for $|P_n^d|$. 

\section{Recursive Formula for $|P_n^d|$} 
We first investigate the problem of finding a formula for $|P_n^d|$. Let us consider the one-dimensional case first. In this case, random variable $X$ takes value $y$ in the set $Y = \{+1,-1\}$. For the illustrative purpose, we may represent the $\mathbb{Z}$ random walk in $\mathbb{Z}^2$ as shown in Figure 1, to show the possible positions after certain numbers of steps:
\begin{center}
\begin{tikzpicture}
\draw (-6,-6) -- (0,0)  -- (6,-6);
\draw (-5,-5) -- (-4,-6) -- (1,-1);
\draw (-4,-4) -- (-2,-6) -- (2,-2);
\draw (-3,-3) -- (0,-6) -- (3,-3);
\draw (-2,-2) -- (2,-6) -- (4,-4);
\draw (-1,-1) -- (4,-6) -- (5,-5);
\draw (0,0) node [anchor=south] {0};
\draw (-1/2,-1/2) node [anchor=east] {-1};
\draw (1/2,-1/2) node [anchor=west] {+1};
\draw (-1,-1) node [anchor=east] {-1};
\draw (1,-1) node [anchor=west] {1};
\draw (-3/2,-3/2) node [anchor=east] {-1};
\draw (3/2,-3/2) node [anchor=west] {+1};
\draw (-1/2,-3/2) node [anchor=east] {+1};
\draw (1/2,-3/2) node [anchor=west] {-1};
\draw (-2,-2) node [anchor=east] {-2};
\draw (0,-2) node [anchor=west] {0};
\draw (2,-2) node [anchor=west] {2};
\draw (-5/2,-5/2) node [anchor=east] {-1};
\draw (5/2,-5/2) node [anchor=west] {+1};
\draw (-3/2,-5/2) node [anchor=east] {+1};
\draw (3/2,-5/2) node [anchor=west] {-1};
\draw (-1/2,-5/2) node [anchor=east] {-1};
\draw (1/2,-5/2) node [anchor=west] {+1};
\draw (-3,-3) node [anchor=east] {-3};
\draw (-1,-3) node [anchor=west] {-1};
\draw (1,-3) node [anchor=west] {1};
\draw (3,-3) node [anchor=west] {3};
\draw (-7/2,-7/2) node [anchor=east] {-1};
\draw (7/2,-7/2) node [anchor=west] {+1};
\draw (-5/2,-7/2) node [anchor=east] {+1};
\draw (5/2,-7/2) node [anchor=west] {-1};
\draw (-3/2,-7/2) node [anchor=east] {-1};
\draw (3/2,-7/2) node [anchor=west] {+1};
\draw (-1/2,-7/2) node [anchor=east] {+1};
\draw (1/2,-7/2) node [anchor=west] {-1};
\draw (-4,-4) node [anchor=east] {-4};
\draw (-2,-4) node [anchor=west] {-2};
\draw (0,-4) node [anchor=west] {0};
\draw (2,-4) node [anchor=west] {2};
\draw (4,-4) node [anchor=west] {4};
\draw (-9/2,-9/2) node [anchor=east] {-1};
\draw (9/2,-9/2) node [anchor=west] {+1};
\draw (-7/2,-9/2) node [anchor=east] {+1};
\draw (7/2,-9/2) node [anchor=west] {-1};
\draw (-5/2,-9/2) node [anchor=east] {-1};
\draw (5/2,-9/2) node [anchor=west] {+1};
\draw (-3/2,-9/2) node [anchor=east] {+1};
\draw (3/2,-9/2) node [anchor=west] {-1};
\draw (-1/2,-9/2) node [anchor=east] {-1};
\draw (1/2,-9/2) node [anchor=west] {+1};
\draw (-5,-5) node [anchor=east] {-5};
\draw (-3,-5) node [anchor=west] {-3};
\draw (-1,-5) node [anchor=west] {-1};
\draw (1,-5) node [anchor=west] {1};
\draw (3,-5) node [anchor=west] {3};
\draw (5,-5) node [anchor=west] {5};
\draw (-11/2,-11/2) node [anchor=east] {-1};
\draw (11/2,-11/2) node [anchor=west] {+1};
\draw (-9/2,-11/2) node [anchor=east] {+1};
\draw (9/2,-11/2) node [anchor=west] {-1};
\draw (-7/2,-11/2) node [anchor=east] {-1};
\draw (7/2,-11/2) node [anchor=west] {+1};
\draw (-5/2,-11/2) node [anchor=east] {+1};
\draw (5/2,-11/2) node [anchor=west] {-1};
\draw (-3/2,-11/2) node [anchor=east] {-1};
\draw (3/2,-11/2) node [anchor=west] {+1};
\draw (-1/2,-11/2) node [anchor=east] {+1};
\draw (1/2,-11/2) node [anchor=west] {-1};
\draw (-6,-6) node [anchor=east] {-6};
\draw (-4,-6) node [anchor=west] {-4};
\draw (-2,-6) node [anchor=west] {-2};
\draw (0,-6) node [anchor=west] {0};
\draw (2,-6) node [anchor=west] {2};
\draw (4,-6) node [anchor=west] {4};
\draw (6,-6) node [anchor=west] {6};
\end{tikzpicture}\\
\ldots \,\,\, \ldots \,\,\, \ldots\\
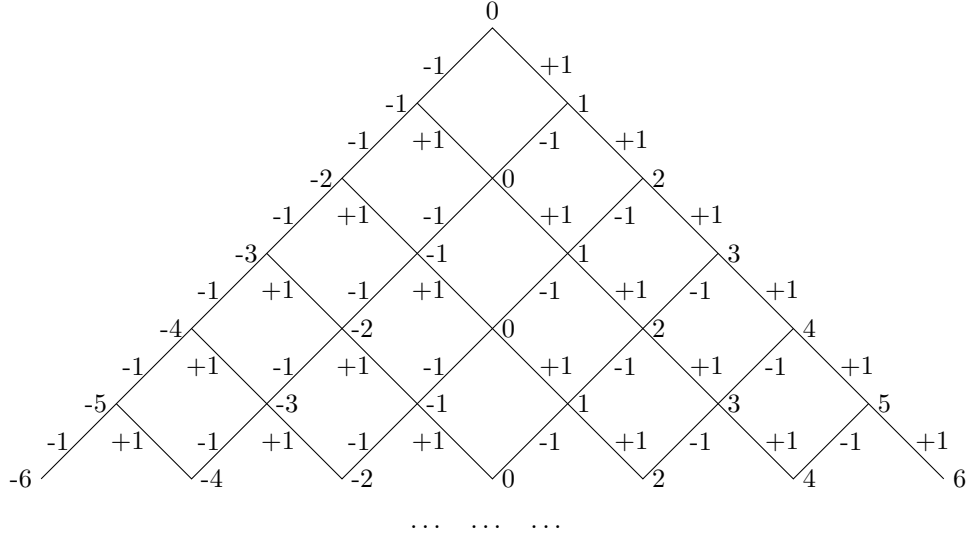
\title{Figure 1: Representation of $\mathbb{Z}$ Random Walk in $\mathbb{Z}^2$}
\end{center}
Clearly, $|P_n^1| = n+1$. 
Now, suppose we want to see how many different positions can be reached in $\mathbb{Z}^2$ in $n$ steps. After $1$ step, there are $4$ possible positions. After $2$ steps, there are $9$ possible positions. We can see this by partitioning $\mathbb{Z}^2$ into "slices" of $\mathbb{Z}$: on the $x$-axis there are $|P_2^1|=3$ possible positions; on either $y=1$ or $y=-1$ there are $|P_1^1|=2$ possible positions; on either $y=2$ or $y=-2$ there is $|P_0^1|=1$ possible position. Therefore 
$$ |P_2^2| = |P_2^1| + 2|P_1^1| + 2|P_0^1| = 3 + 2\cdot 2 + 2\cdot 1 = 9$$
\begin{center}
\begin{tikzpicture}
\draw (-3,0) -- (3,0);
\draw (0,3) -- (0,-3);
\draw (-3,1) -- (3,1);
\draw (3,1) node [anchor=west] {$y=1$};
\draw (-3,2) -- (3,2);
\draw (3,2) node [anchor=west] {$y=2$};
\draw (-3,-1) -- (3,-1);
\draw (3,-1) node [anchor=west] {$y=-1$};
\draw (-3,-2) -- (3,-2);
\draw (3,-2) node [anchor=west] {$y=-2$};
\draw (3,0) node [anchor=west] {$x(y=0)$};
\draw (0,3) node [anchor=east] {$y$};
\filldraw (-2,0) [black] circle (2pt);
\draw (-2,0) node [anchor=north] {$(-2,0)$};
\filldraw (0,0) [black] circle (2pt);
\draw (-0.5,0) node [anchor=north] {$(0,0)$};
\filldraw (2,0) [black] circle (2pt);
\draw (2,0) node [anchor=north] {$(2,0)$};
\filldraw (-1,1) [black] circle (2pt);
\draw (-1,1) node [anchor=north] {$(-1,1)$};
\filldraw (1,1) [black] circle (2pt);
\draw (1,1) node [anchor=north] {$(1,1)$};
\filldraw (-1,-1) [black] circle (2pt);
\draw (-1,-1) node [anchor=north] {$(-1,-1)$};
\filldraw (1,-1) [black] circle (2pt);
\draw (1,-1) node [anchor=north] {$(1,-1)$};
\filldraw (0,2) [black] circle (2pt);
\draw (-0.5,2) node [anchor=north] {$(0,2)$};
\filldraw (0,-2) [black] circle (2pt);
\draw (-0.5,-2) node [anchor=north] {$(0,-2)$};
\draw (-3,0) node [anchor=east] {$|P_2^1|$};
\draw (-3,1) node [anchor=east] {$|P_1^1|$};
\draw (-3,2) node [anchor=east] {$|P_0^1|$};
\draw (-3,-1) node [anchor=east] {$|P_1^1|$};
\draw (-3,-2) node [anchor=east] {$|P_0^1|$};
\end{tikzpicture}\\
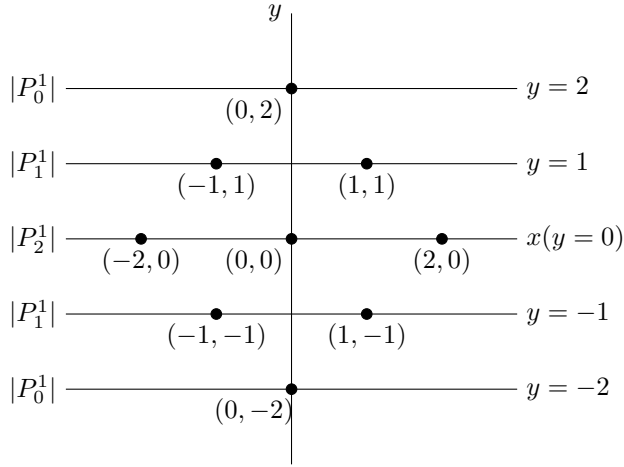
\title{Figure 2: Slicing Partitions of $P^2_2$}
\end{center}
This idea motivates us to generalize. Every position in $\mathbb{Z}^{d+1}$ can be written as a $(d+1)\times1$ matrix where each entry is in $\mathbb{Z}$. We consider the possible positions after $n$ steps. Since each step is a unit vector, the maximum absolute value of the last entry is $n$. Therefore, the last entry takes value in the set
$$\{-n,-(n-1),\ldots,-2,-1,0,1,2,\ldots, (n-1), n\}$$
First, consider the positions in which the last entry is $n$ or $-n$. In each case the number of possible positions is $|P_0^d|$ (since there can be no other steps in any other dimension). Next, consider the positions in which the last entry is $(n-1)$ or $-(n-1)$. In each case the number of possible positions is $|P_1^d|$ (there is one step in one of the remaining $d$ dimensions). We do this until we get to the positions in which the last entry is $1$ or $-1$. In each case the number of possible positions is $|P_{n-1}^d|$, since the other $(n-1)$ steps are in the remaining $d$ dimensions. Finally, if the last entry is $0$, the number of possible positions is $|P_n^d|$. Therefore we have the following relation:
$$ |P_n^{d+1}| = 2|P_0^d| + 2|P_1^d| + \ldots + 2|P_{n-1}^d| + |P_n^d|$$
which, if written more simply, becomes
\begin{equation}
|P_n^{d+1}| = |P_n^d| + 2\bigg(\sum_{k=0}^{n-1} |P_k^d| \bigg) = 2\bigg(\sum_{k=0}^n |P_k^d|\bigg) - |P_n^d|
\end{equation}
Using this recurrence relation, we can calculate the following
\begin{align*}
|P_n^2| & = 2\bigg(\sum_{k=0}^n |P_k^1|\bigg) - |P_n^1| \\
            & = 2\sum_{k=0}^n (k+1) - (n+1) \\
            & = 2\sum_{k=0}^n k + 2\sum_{k=0}^n 1 - (n+1) \\
            & = 2\frac{n(n+1)}{2} + 2(n+1) - (n+1) \\
            & = n(n+1) + (n+1) \\
            & = (n+1)^2
\end{align*}
We may guess that the pattern $(n+1)^d$ continues, but unfortunately this is not the case:
\begin{align*}
|P_n^3| & = 2\bigg(\sum_{k=0}^n |P_k^2|\bigg) - |P_n^2| \\
            & = 2\sum_{k=0}^n (k+1)^2 - (n+1)^2 \\
            & = 2\sum_{k=1}^{n+1} k^2 - (n+1)^2 \\
            & = 2\frac{(n+1)(n+1+1)(2(n+1)+1)}{6} - (n+1)^2 \\
            & = \frac{(n+1)(n+2)(2n+3)-3(n+1)^2}{3} \\
            & = \frac{(n+1)(2n^2 + 7n + 6 - 3(n+1))}{3} \\
            & = \frac{(n+1)(2n^2+4n+3)}{3} \\
            & = \frac{2n^3+6n^2+7n+3}{3} \\
            & = \frac{2}{3}n^3 + 2n^2 + \frac{7}{3}n + 1
\end{align*}
This suggests that $|P_n^d|$ has a more complicated form. However, one of the things we noticed is that $|P_n^d|$ seems to always be a polynomial in $n$ with degree $d$. Indeed, this can be verified by induction.
\section{$|P_n^d|$ from generating functions}
Our first approach to compute $|P^d_n|$ involves generating functions and our recursive formula $(1)$.
Our goal was to find a generating function $g^d(x)=|P^d_0|x^0+|P^d_1|x^1+|P^d_2|x^2 \cdots =\sum_{i=0}^{\infty}|P^d_i|x^i$ and to determine a formula to find $|P^d_n|$ by expanding $g^d(x)$.
First, we desire to know the formula for $g^1(x)=|P^1_0|x^0+|P^1_1|x^1+\cdots = 1+2x+3x^2+\cdots$.  We know that this is the derivative of $\frac{1}{1-x}=1+x+x^2+x^3\cdots$. Therefore, we have
\begin{equation*}
g^1(x)=\frac{1}{(1-x)^2}  
\end{equation*}
Furthermore, we have a recurrence relation that gives terms of $|P^d_n|$ in terms of $|P^{d-1}_i|$ where $i$ is a natural number ranging from 0 to $n$.  Thus, it might be natural to try to find out what $g^{d+1}(x)$ is in terms of $g^{d}(x)$. In fact, we may do just that.  

\begin{align*}
g^{d}(x)=&|P^d_0|x^0+\,\,|P^d_1|x^1+\,\,|P^d_2|x^2+\,\,|P^d_3|x^3\cdots\\
+2xg^{d}(x)=& \qquad \enspace  +2|P^d_0|x^1+2|P^d_1|x^2+2|P^d_2|x^3\cdots\\
+2x^2g^{d}(x)=& \qquad \quad \enspace \qquad \quad \,\,\,+2|P^d_0|x^2+2|P^d_1|x^3\cdots\\
\vdots&
\end{align*}
We can see that the coefficients of each expanded term up to $2x^n g^d(x)$ sums to $|P^{d+1}_n|$.  Therefore, $$g^{d+1}(x)=g^{d}(x)+\sum^{\infty}_{i=1}2x^i g^{d}(x)$$
$$g^{d+1}(x)=g^{d}(x) \big(1+2\sum^{\infty}_{i=1}x^i \big)$$
$$g^{d+1}(x)=g^{d}(x) \bigg(1+2\frac{x}{1-x} \bigg)$$
$$g^{d+1}(x)=g^{d}(x) \bigg(\frac{1+x}{1-x} \bigg)$$
Given our initial condition, we find:
\begin{equation}
g^d(x)=\frac{(1+x)^{d-1}}{(1-x)^{d+1}}
\end{equation}
Next, we must expand this function back into the form of an infinite series and compute the coefficients of each term of the function. We will examine this function in two parts: the binomial polynomial in the numerator, and the infinite series generated by the denominator. Using the binomial expansion theorem:
\begin{equation*}
(1+x)^{d-1}=\displaystyle\sum^{d-1}_{l=0}\binom{d-1}{l}x^l
\end{equation*}
We now turn to the denominator:
\begin{equation*}
    \frac{1}{(1-x)^{d+1}}
\end{equation*}
\noindent We will manipulate this into another form using our previous background in generating functions.
\begin{equation*}
    \frac{1}{(1-x)^{d+1}}=\left( \frac{1}{1-x}\right)^{d+1}=(1+x+x^2+\cdots)^{d+1}
\end{equation*}
\begin{equation*}
    =\underbrace{(1+x+x^2+\cdots)(1+x+x^2+\cdots)\dots (1+x+x^2+\cdots)}_\text{$d+1$ times}
\end{equation*}

\noindent To determine the expansion of the denominator of the function,  we must order and label the $d+1$ polynomials of the form $(1+x+x^2+\cdots)$.  We are able to determine the coefficients of this infinite series by formalizing the process we use to expand it.  We choose one term of the form $x^{e_i}$ from each of the polynomials of the form $(1+x+x^2+\cdots)$. Clearly, $e_i$ is always a nonnegative integer.  We express this polynomial as the product $\prod^{d+1}_{i=1}x^{e_i}$
for all nonnegative integers $e_i\geq 0$.  The coefficient of the $x^m$ term is the number of terms of the form $x^{e_1}x^{e_2}\cdots x^{e_{d+1}}=x^m$.  By exponent laws, we get $x^{e_1+e_2+\cdots+e_{d+1}}=x^m$.  This is equivalent to asking how many $e_1+e_2+\cdots+e_{d+1}=m$ when $\forall i, \; e_i\geq0$.  It is a well known combinatorial result that this is $\displaystyle\binom{d+m}{d}$.  Therefore, $$\displaystyle\frac{1}{(1-x)^{d+1}}=\displaystyle\sum^{\infty}_{m=0}\binom{d+m}{d}x^m$$\\

\noindent To derive our final expansion, we multiply
\begin{equation*}
   (1+x)^{d-1} \frac{1}{(1-x)^{d+1}}=\displaystyle\sum^{d-1}_{l=0}\binom{d-1}{l}x^l\displaystyle\sum^{\infty}_{m=0}\binom{d+m}{d}x^m
\end{equation*}

\noindent To determine the coefficient $|P^d_n|$ for any $n$, we must sum over all products of terms $x^l x^m=x^{l+m}=x^n$.  After calculating the expansion of each coefficient, we obtain the final expansion of $|P^d_n|$ using generating functions:
\begin{equation}
    |P^d_n|=\mathlarger{\mathlarger{\sum}}^{d-1}_{k=0}\binom{d-1}{k}\binom{d+n-k}{d}
\end{equation}

\section{$|P_n^d|$ from Faulhaber's formula}
Next, we present another method to find $|P_n^d|$. Since we know $|P_n^d|$ is a polynomial in $n$ with degree $d$, we only need the coefficients to fully determine $|P_n^d|$. We may assume
\begin{equation}
|P_n^d| = c(d,d)n^d + c(d,d-1)n^{d-1} + \ldots + c(d,1)n + c(d,0) = \mathbf{c}_d \cdot \mathbf{n}_d
\end{equation}
where
\begin{equation}
\mathbf{c}_d =
\begin{pmatrix}
c(d,d)\\c(d,d-1)\\ \vdots \\ c(d,1) \\ c(d,0)
\end{pmatrix},
\mathbf{n}_d =
\begin{pmatrix}
n^d \\ n^{d-1} \\ \vdots \\ n \\ 1
\end{pmatrix}
\end{equation}
Using this notation, we may rewrite recurrence relation $(6)$ as:
$$c(d+1,d+1)n^{d+1} + c(d+1,d)n^d + \ldots + c(d+1,1)n + c(d+1,0)$$
\begin{equation*}
    \begin{split}
& = 2 \sum_{k=0}^n \big(c(d,d)k^d + c(d,d-1)k^{d-1} + \ldots + c(d,1)k + c(d,0) \big) \\
& - \big(c(d,d)n^d + c(d,d-1)n^{d-1} + \ldots + c(d,1)n + c(d,0) \big) \\
& = 2c(d,d)\sum_{k=0}^n k^d + 2c(d,d-1)\sum_{k=0}^n k^{d-1} + \ldots + 2c(d,1)\sum_{k=0}^n k + 2c(d,0)\sum_{k=0}^n 1 \\
& - \big(c(d,d)n^d + c(d,d-1)n^{d-1} + \ldots + c(d,1)n + c(d,0) \big) 
    \end{split}
\end{equation*}
We notice that sums of the form $\displaystyle \sum_{k=0}^n k^d$ appear. The explicit expansion of this sum is given by Faulhaber's formula:
\begin{equation}
\sum_{k=0}^n k^d = 1^d + 2^d + \ldots + n^d = \frac{n^{d+1}}{d+1} + \frac{1}{2}n^d + \frac{1}{d+1}\sum_{k=2}^d B_k \binom{d+1}{k} n^{d-k+1}
\end{equation}
where $B_k$ represents the $k$th Bernoulli number. We use $(6)$ to find the relationship between $\mathbf{c}_{d+1}$ and $\mathbf{c}_d$:
\begin{equation*}
   \begin{split}
& c(d+1,d+1)n^{d+1} + c(d+1,d)n^d + \ldots + c(d+1,1)n + c(d+1,0) \\
& = 2c(d,d)\sum_{k=0}^n k^d + 2c(d,d-1)\sum_{k=0}^n k^{d-1} + \ldots + 2c(d,1)\sum_{k=0}^n k + 2c(d,0)\sum_{k=0}^n 1 \\
& - \big(c(d,d)n^d + c(d,d-1)n^{d-1} + \ldots + c(d,1)n + c(d,0)\big) \\
& = 2c(d,d) \bigg(\frac{n^{d+1}}{d+1} + \frac{1}{2}n^d + \frac{1}{d+1}\sum_{k=2}^d B_k \binom{d+1}{k} n^{d-k+1} \bigg) \\
& +2c(d,d-1) \bigg(\frac{n^d}{d} + \frac{1}{2}n^{d-1} + \frac{1}{d} \sum_{k=2}^{d-1} B_k \binom{d}{k} n^{d-k+1} \bigg) \\
& + \ldots \\
& +2c(d,1) \bigg(\frac{n^2}{2} + \frac{1}{2}n \bigg) + 2c(d,0) (n+1) \\
& -c(d,d)n^d - c(d,d-1)n^{d-1} - \ldots - c(d,1)n -c(d,0) \\
& = 2c(d,d) \bigg(\frac{n^{d+1}}{d+1} + \frac{1}{d+1} \sum_{k=2}^d B_k \binom{d+1}{k} n^{d-k+1} \bigg) \\
& + 2c(d,d-1) \bigg(\frac{n^d}{d} + \frac{1}{d} \sum_{k=2}^{d-1} B_k \binom{d}{k} n^{d-k} \bigg) \\
& + 2c(d,d-2) \bigg(\frac{n^{d-1}}{d-1} + \frac{1}{d-1} \sum_{k=2}^{d-2} B_k \binom{d-1}{k} n^{d-k-1} \bigg) \\
& + 2c(d,d-3) \bigg(\frac{n^{d-2}}{d-2} + \frac{1}{d-2} \sum_{k=2}^{d-3} B_k \binom{d-2}{k} n^{d-k-2} \bigg) \\
& + \ldots \\
& + 2c(d,4) \bigg(\frac{n^5}{5} + \frac{1}{5} \sum_{k=2}^4 B_k \binom{5}{k} n^{5-k} \bigg) \\
& + 2c(d,3) \bigg(\frac{n^4}{4} + \frac{1}{4} \sum_{k=2}^3 B_k \binom{4}{k} n^{4-k} \bigg) \\
& + 2c(d,2) \bigg(\frac{n^3}{3} + \frac{1}{3} \sum_{k=2}^2 B_k \binom{3}{k} n^{3-k} \bigg) \\
& + 2c(d,1) \bigg(\frac{n^2}{2} \bigg) + 2c(d,0) (n+1) - c(d,0)
    \end{split}
\end{equation*}
Comparing coefficients of $n^{d+1}, n^d, \ldots, n, 1$, we have the following:
\begin{align*}
c(d+1,d+1) & = \frac{2c(d,d)}{d+1} \\
c(d+1,d) & = \frac{2c(d,d-1)}{d} \\
c(d+1,d-1) & = \frac{2B_2}{d+1} \binom{d+1}{2} c(d,d) + \frac{2c(d,d-2)}{d-1} \\
c(d+1,d-2) & = \frac{2B_3}{d+1} \binom{d+1}{3} c(d,d) + \frac{2B_2}{d} \binom{d}{2} c(d,d-1) + \frac{2c(d,d-3)}{d-2} \\
& \vdots \\
c(d+1,d-j) & = \frac{2B_{j+1}}{d+1} \binom{d+1}{j+1} c(d,d) + \frac{2B_j}{d} \binom{d}{j} c(d,d-1) + \frac{2B_{j-1}}{d-1} \binom{d-1}{j-1} c(d,d-2) + \ldots \\
& +\frac{2B_2}{d-j+2} \binom{d-j+2}{2} c(d,d-j+1) + \frac{2c(d,d-j-1)}{d-j} \\
& \vdots \\
& + \frac{2B_2}{5} \binom{5}{2} c(d,4) + \frac{2c(d,2)}{3} \\
c(d+1,2) & = \frac{2B_{d-1}}{d+1} \binom{d+1}{d-1} c(d,d) + \frac{2B_{d-2}}{d} \binom{d}{d-2} c(d,d-1) + \frac{2B_{d-3}}{d-1} \binom{d-1}{d-3} c(d,d-2) + \ldots \\
& + \frac{2B_3}{5} \binom{5}{3} c(d,4) + \frac{2B_2}{4} \binom{4}{2} c(d,3) + \frac{2c(d,1)}{2} \\
c(d+1,1) & = \frac{2B_d}{d+1} \binom{d+1}{d} c(d,d) + \frac{2B_{d-1}}{d} \binom{d}{d-1} c(d,d-1) + \frac{2B_{d-2}}{d-1} \binom{d-1}{d-2} c(d,d-2) + \ldots \\
& + \frac{2B_4}{5} \binom{5}{4} c(d,4) + \frac{2B_3}{4} \binom{4}{3} c(d,3) + \frac{2B_2}{3} \binom{3}{2} c(d,2) + 2c(d,0) \\
c(d+1,0) & = c(d,0)
\end{align*}
Using matrix form we may write this more compactly:
\begin{equation}
\mathbf{c}_{d+1} = M_d^{d+1} \mathbf{c}_d
\end{equation}
where
\begin{equation*}
M_d^{d+1} = 
\begin{pmatrix}
\frac{2}{d+1} & 0 & 0 & 0 & 0 & \cdots & 0 & \cdots & 0  & 0 & 0 
\\
0 & \frac{2}{d} & 0 & 0 & 0 & \cdots & 0  & \cdots & 0 & 0 & 0 
\\
\frac{2B_2}{d+1}\binom{d+1}{2} & 0 & \frac{2}{d-1} & 0 & 0 & \cdots & 0  & \cdots & 0 & 0 & 0   
\\
\frac{2B_3}{d+1}\binom{d+1}{3} & \frac{2B_2}{d}\binom{d}{2} & 0 & \frac{2}{d-2} & 0 & \cdots & 0  & \cdots & 0 & 0 & 0                                         
\\
\frac{2B_4}{d+1}\binom{d+1}{4} & \frac{2B_3}{d}\binom{d}{3} & \frac{2B_2}{d-1}\binom{d-1}{2} & 0 & \frac{2}{d-3} & \cdots & 0 & \cdots & 0 & 0 & 0 
\\
\vdots & \vdots & \vdots & \vdots & \vdots & \ddots & \vdots & & \vdots & \vdots & \vdots 
\\
\frac{2B_{j+1}}{d+1}\binom{d+1}{j+1} & \frac{2B_j}{d}\binom{d}{j} & \frac{2B_{j-1}}{d-1}\binom{d-1}{j-1} & \cdots & \frac{2B_2}{d-j+2}\binom{d-j+2}{2} & 0 & \frac{2}{d-j} & \cdots & 0 & 0 & 0 
\\
\vdots & \vdots & \vdots & \vdots & \vdots & \vdots & \vdots & \ddots & \vdots & \vdots & \vdots 
\\
\frac{2B_{d-2}}{d+1}\binom{d+1}{d-2} & \frac{2B_{d-3}}{d}\binom{d}{d-3} & \frac{2B_{d-4}}{d-1}\binom{d-1}{d-4} & \cdots & \cdots & \cdots & \frac{2B_2}{5}\binom{5}{2} & 0 & \frac{2}{3} & 0 & 0 
\\
\frac{2B_{d-1}}{d+1}\binom{d+1}{d-1} & \frac{2B_{d-2}}{d}\binom{d}{d-2} & \frac{2B_{d-3}}{d-1}\binom{d-1}{d-3} & \cdots & \cdots & \cdots & \cdots & \frac{2B_2}{4}\binom{4}{2} & 0 & \frac{2}{2} & 0 
\\
\frac{2B_d}{d+1}\binom{d+1}{d} & \frac{2B_{d-1}}{d}\binom{d}{d-1} & \frac{2B_{d-2}}{d-1}\binom{d-1}{d-2} & \cdots & \cdots & \cdots & \frac{2B_4}{5}\binom{5}{4} & \frac{2B_3}{4}\binom{4}{3} & \frac{2B_2}{3}\binom{3}{2} & 0 & 2 
\\
0 & 0 & 0 & 0 & 0 & \cdots & 0 & \cdots & 0 & 0 & 1
\end{pmatrix}
\end{equation*}
is a $(d+2)\times (d+1)$ matrix. \\
For example, 
\begin{align*}
\mathbf{c}_1 & = 
\begin{pmatrix}
c(1,1) \\ c(1,0)
\end{pmatrix}
= 
\begin{pmatrix}
1 \\ 1
\end{pmatrix} \\
\mathbf{c}_2 & = M_1^2 \mathbf{c}_1 =
\begin{pmatrix}
1 & 0 \\
0 & 2 \\
0 & 1 \\
\end{pmatrix}
\begin{pmatrix}
1 \\ 1
\end{pmatrix} =
\begin{pmatrix}
1 \\ 2 \\ 1
\end{pmatrix} \\
\mathbf{c}_3 & = M_2^3 \mathbf{c}_2 = 
\begin{pmatrix}
\frac{2}{3} & 0 & 0 \\
0 & 1 & 0 \\
\frac{1}{3} & 0 & 2 \\
0 & 0 & 1 \\
\end{pmatrix}
\begin{pmatrix}
1 \\ 2 \\ 1
\end{pmatrix} = 
\begin{pmatrix}
\frac{2}{3} \\ 2 \\ \frac{7}{3} \\ 1
\end{pmatrix} \\
\mathbf{c}_4 & = M_3^4 \mathbf{c}_3 =
\begin{pmatrix}
\frac{1}{2} & 0 & 0 & 0 \\
0 & \frac{2}{3} & 0 & 0 \\
\frac{1}{2} & 0 & 1 & 0 \\
0 & \frac{1}{3} & 0 & 2 \\
0 & 0 & 0 & 1 \\
\end{pmatrix}
\begin{pmatrix}
\frac{2}{3} \\ 2 \\ \frac{7}{3} \\ 1
\end{pmatrix}=
\begin{pmatrix}
1/3 \\ 4/3 \\ 8/3 \\ 8/3 \\ 1
\end{pmatrix}
\end{align*}
which gives, correspondingly,
\begin{align*}
|P_n^1| & = n + 1 \\
|P_n^2| & = n^2 + 2n + 1 \\
|P_n^3| & = \frac{2}{3} n^3 + 2n^2 + \frac{7}{3} n + 1 \\
|P_n^4| & = \frac{1}{3} n^4 + \frac{4}{3} n^3 + \frac{8}{3} n^2 + \frac{8}{3} n + 1 
\end{align*}
This is summarized in the following theorem:
\begin{theorem}
The cardinality of the set of possible positions in $\mathbb{Z}^d$ after a $n$-walk, $|P_n^d|$, is given by
\begin{equation}
|P_n^d| = \mathbf{c}_d \cdot \mathbf{n}_d
\end{equation}
where $\mathbf{c}_d$ and $\mathbf{n}_d$ are as given in $(11)$, and $\mathbf{c}_d$ is calculated recursively by
\begin{equation}
\mathbf{c}_d = M_{d-1}^d M_{d-2}^{d-1} \cdots M_2^3 M_1^2 \mathbf{c}_1
\end{equation}
where
\begin{equation*}
\mathbf{c}_1 = 
\begin{pmatrix}
1 \\ 
1 \\
\end{pmatrix}
\end{equation*}
and $M_d^{d+1}$ is given above.
\end{theorem}
\section{Comparing $c(d,d), c(d,d-1), c(d,d-2)$ from two approaches}
Since our two methods of finding $|P_n^d|$ both originate from the recurrence relation (1), they must yield the same result for any nonnegrative integer $d$. In fact, this was proved to be true from $d=1$ to $5$. During the process of proving this, we found that although our second method involving matrices seems to be complicated and does not provide a formula for $|P_n^d|$, when actually carrying out calculation, it is much simpler than expanding the sum obtained by our first method, as only basic arithmetic of rational numbers is involved. We try to derive explicitly from both methods the coefficient of each power of $n$ in $|P_n^d|$ and show that they are equal. Unfortunately, this turns out to be difficult for coefficient of the most general form $c(d,d-j)$ in $|P_n^d|$ and we pose this as an open problem. However, for $c(d,d),c(d,d-1)$ and $c(d,d-2)$ we are able to prove that both methods give the same result,as we now show.\\
\subsection{$c(d,d)$}
The first row of $M_d^{d+1}$ tells us
\begin{equation*}
    c(d+1,d+1) = \frac{2}{d+1} c(d,d)
\end{equation*}
which, if we develop further, yields
\begin{align*}
    c(d+1,d+1) & = \frac{2}{d+1} c(d,d) \\
               & = \frac{2}{d+1} \frac{2}{d} c(d-1,d-1) \\
               & = \cdots \\
               & = \frac{2}{d+1} \frac{2}{d} \cdots \frac{2}{2} c(1,1) \\
               & = \frac{2^d}{(d+1)!}
\end{align*}
Therefore, 
\begin{equation}
    c(d,d) = \frac{2^{d-1}}{d\,!}
\end{equation}
Expanding $(1)$ gives
\begin{equation}
    |P_n^d| = \mathlarger{\sum}_{k=0}^{d-1} \binom{d-1}{k} \frac{(n-k+d)\,!}{d\,! \, (n-k)\,!} =
    \frac{1}{d\,!} \mathlarger{\sum}_{k=0}^{d-1} \binom{d-1}{k} (n-(k-1)) (n-(k-2)) \cdots (n-(k-d))
\end{equation}
From this, and by the binomial theorem, the coefficient of $n^d$ is
\begin{equation*}
    c(d,d) = \frac{1}{d\,!} \mathlarger{\sum}_{k=0}^{d-1} \binom{d-1}{k} = \frac{2^{d-1}}{d\,!}
\end{equation*}
which agrees with $(10)$.
\subsection{$c(d,d-1)$}
The second row of $M_d^{d+1}$ tells us
\begin{equation*}
    c(d+1,d) = \frac{2}{d} c(d,d-1)
\end{equation*}
which we may develop as before:
\begin{align*}
    c(d+1,d) & = \frac{2}{d} c(d,d-1) \\
             & = \frac{2}{d} \frac{2}{d-1} c(d-1,d-2) \\
             & = \cdots \\
             & = \frac{2}{d} \frac{2}{d-1} \cdots \frac{2}{2} \frac{2}{1} c(1,0) \\
             & = \frac{2^d}{d\,!}
\end{align*}
and thus
\begin{equation}
    c(d,d-1) = \frac{2^{d-1}}{(d-1)\,!}
\end{equation}
Now we look at $(19)$ again and try to determine the coefficient of $n^{d-1}$. To contribute to $c(d,d-1)n^{d-1}$, we may take $-(k-i)$ out of the $i$th factor and $n$ out of the other $(d-1)$ factors; since $i$ ranges from $1$ to $d$, there are $d$ ways to do this, and we can write 
\begin{align*}
    c(d,d-1) & = \frac{1}{d\,!} \mathlarger{\sum}_{k=0}^{d-1} \binom{d-1}{k} \mathlarger{\sum}_{i=1}^d(-(k-i)) \\
             & = \frac{1}{d\,!} \mathlarger{\sum}_{k=0}^{d-1} \binom{d-1}{k}\bigg( \frac{d(d+1)}{2} - kd \bigg) \\
             & = \frac{d(d+1)2^{d-1}}{2d\,!} - \frac{1}{(d-1)\,!} \mathlarger{\sum}_{k=0}^{d-1} k\binom{d-1}{k} \\
             & = \frac{(d+1)2^{d-2}}{(d-1)\,!} - \frac{2^{d-2}}{(d-2)\,!} \\
             & = \frac{2^{d-2}}{(d-2)\,!} \bigg(\frac{d+1}{d-1}-1 \bigg) \\
             & = \frac{2^{d-1}}{(d-1)\,!}
\end{align*}
(In the process we used the identity $k\displaystyle\binom{n}{k} = n\displaystyle\binom{n-1}{k-1}$)\\
which agrees with $(12)$.
\subsection{$c(d,d-2)$}
The third row of $M_d^{d+1}$ tells us
\begin{equation*}
    c(d+1,d-1) = \frac{2B_2}{d+1}\binom{d+1}{2}c(d,d) + \frac{2}{d-1}c(d,d-2)
\end{equation*}
which we may develop further to obtain
\begin{align*}
    c(d+1,d-1) & = \frac{d}{6}c(d,d) + \frac{2}{d-1}c(d,d-2) \\
               & = \frac{2^{d-1}}{6(d-1)\,!} + \frac{2}{d-1}c(d,d-2) \\
               & = \cdots \\
               & = (d-1)\frac{2^{d-1}}{6(d-1)\,!} + \frac{2^{d-1}}{(d-1)(d-2)\cdots(d-(d-1))}c(2,0) \\
               & = \frac{2^{d-1}}{6(d-2)\,!} + \frac{2^{d-1}}{(d-1)\,!} \\
               & = \frac{2^{d-1}}{6(d-2)\,!}\bigg(1+\frac{6}{d-1}\bigg) \\
               & = \frac{2^{d-1}(d+5)}{6(d-1)\,!}
\end{align*}
and thus
\begin{equation}
    c(d,d-2) = \frac{2^{d-2}(d+4)}{6(d-2)\,!}
\end{equation}
Before we turn to finding $c(d,d-2)$ from $(11)$, we first prove three lemmas which are going to be used.
\begin{lemma}
$\displaystyle\sum_{k=1}^n k\binom{n}{k} = n2^{n-1}$ and $\displaystyle\sum_{k=1}^n k^2\binom{n}{k} = n(n+1)2^{n-2}$
\end{lemma}
\begin{proof}
We begin with a special case of the Binomial Theorem:
\begin{equation}
    (1+x)^n = \sum_{k=0}^n \binom{n}{k}x^k
\end{equation}
Differentiating both sides of $(14)$ with respect to $x$, we get
\begin{equation}
    n(1+x)^{n-1} = \sum_{k=1}^n k\binom{n}{k}x^{k-1}
\end{equation}
Substituting $x=1$ in $(15)$, we get
\begin{equation}
    \sum_{k=1}^n k\binom{n}{k} = n2^{n-1}
\end{equation}
which is the first identity. Multiplying both sides of $(15)$ by $x$, we get
\begin{equation}
    nx(1+x)^{n-1} = \sum_{k=1}^n k\binom{n}{k}x^k
\end{equation}
Differentiating both sides of $(17)$ with respect to $x$, we get
\begin{equation}
    n\big((1+x)^{n-1} + (n-1)x(1+x)^{n-2}\big) = \sum_{k=1}^n k^2\binom{n}{k}x^{k-1}
\end{equation}
Substituting $x=1$ in $(18)$ finally gives
\begin{equation}
    \sum_{k=1}^n k^2\binom{n}{k} = n\big(2^{n-1} + (n-1)2^{n-2}\big) = n(n+1)2^{n-2}
\end{equation}
\end{proof}
\begin{lemma}
$\displaystyle\sum_{1\leq i<j \leq n}(i+j) = \frac{n(n+1)(n-1)}{2}$
\end{lemma}
\begin{lemma}
$\displaystyle \sum_{1\leq i<j \leq n} ij = \frac{n(n+1)(3n+2)(n-1)}{24}$
\end{lemma}
We are now ready to proceed to finding $c(d,d-2)$ from $(11)$. 
\begin{align*}
    c(d,d-2) 
          & = \frac{1}{d\,!} \mathlarger{\sum}_{k=0}^{d-1} \binom{d-1}{k} (-1)^2 \mathlarger{\sum}_{1\leq i<j \leq d} (k-i)(k-j) \\
          & = \frac{1}{d\,!} \mathlarger{\sum}_{k=0}^{d-1} \binom{d-1}{k} \mathlarger{\sum}_{1\leq i<j \leq d} \mathlarger{(} k^2 - (i+j)k + ij \mathlarger{)} \\
          & = \frac{1}{d\,!} \mathlarger{\sum}_{k=0}^{d-1} \binom{d-1}{k} \bigg( \binom{d}{2}k^2 - k\cdot \frac{d(d+1)(d-1)}{2} + \frac{d(d+1)(3d+2)(d-1)}{24} \bigg) \\
          & = \frac{1}{d\,!} \bigg( \binom{d}{2}\mathlarger{\sum}_{k=0}^{d-1} k^2\binom{d-1}{k} - \frac{d(d+1)(d-1)}{2} \mathlarger{\sum}_{k=0}^{d-1} k\binom{d-1}{k} + \frac{d(d+1)(3d+2)(d-1)}{24} \mathlarger{\sum}_{k=0}^{d-1} \binom{d-1}{k} \bigg) \\
          & = \frac{2^{d-3}}{6(d-2)\,!} (2d+8) = \frac{2^{d-2}(d+4)}{6(d-2)\,!}
\end{align*}
which agrees with $(13)$.
\subsection{What's next?}
Using similar procedure, from the matrix $M_d^{d+1}$ we can show that
\begin{align}
    c(d,d-3) & = \frac{2^{d-2}d}{(d-3)\,!} \\
    c(d,d-4) & = \frac{2^{d-6}(5d^2+33d-32)}{45(d-4)\,!}
\end{align}
However, if we expand $(19)$, we get the following:
\begin{align}
    c(d,d-3) & = \frac{1}{d\,!} \mathlarger{\sum}_{k=0}^{d-1} \binom{d-1}{k} (-1)^3 \mathlarger{\sum}_{1\leq i_1<i_2<i_3 \leq d} (k-i_1)(k-i_2)(k-i_3) \\
    c(d,d-4) & = \frac{1}{d\,!} \mathlarger{\sum}_{k=0}^{d-1} \binom{d-1}{k}(-1)^4 \mathlarger{\sum}_{1\leq i_1<i_2<i_3<i_4 \leq d} (k-i_1)(k-i_2)(k-i_3)(k-i_4) 
\end{align}
Of course, we have reason to believe that $(22)$ ends up being $(20)$, and $(23)$ ends up being $(21)$. However, a major difficulty appears, which is to evaluate the sums of the form
\begin{equation}
    \mathlarger{\sum}_{1\leq i_1<i_2<\cdots<i_j \leq d} (k-i_1)(k-i_2)\cdots(k-i_j)
\end{equation}
which appears in the expression for $c(d,d-j)$ obtained from $(11)$:
\begin{equation}
    c(d,d-j) = \frac{1}{d\,!}\mathlarger{\sum}_{k=0}^{d-1} \binom{d-1}{k} (-1)^j \mathlarger{\sum}_{1\leq i_1<i_2<\cdots<i_j \leq d} (k-i_1)(k-i_2)\cdots(k-i_j)
\end{equation}
One possible way to simplify the sum $(24)$ is by the substitution
\begin{equation*}
    i_1' = k-i_j, i_2' = k-i_{j-1}, \cdots, i_j' = k-i_1
\end{equation*}
So that 
\begin{equation*}
    \mathlarger{\sum}_{1\leq i_1<i_2<\cdots<i_j \leq d} (k-i_1)(k-i_2)\cdots(k-i_j) =
    \mathlarger{\sum}_{k-d\leq i_1'<i_2'<\cdots<i_j' \leq k-1} i_1'i_2'\cdots i_j'
\end{equation*}
This requires us to do sums of the form
\begin{equation*}
    \mathlarger{\sum}_{a\leq i_1<i_2<\cdots<i_j \leq b} i_1i_2\cdots i_j \,\,\,\,\,(a,b \in \mathbb{Z},a<b)
\end{equation*}
Of course, we have
\begin{align*}
    \mathlarger{\sum}_{a\leq i_1<i_2<\cdots<i_j \leq b} i_1i_2\cdots i_j 
    & = 
    a\mathlarger{\sum}_{a+1 \leq i_2<\cdots<i_j \leq b} i_2\cdots i_j \\
    & + 
    (a+1)\mathlarger{\sum}_{a+2 \leq i_2<\cdots<i_j \leq b} i_2\cdots i_j \\
    & +
    \cdots + (b-j)\mathlarger{\sum}_{b-j+1 \leq i_2<\cdots<i_j \leq b} i_2\cdots i_j \\
    & + (b-j+1)\mathlarger{\sum}_{b-j+2 \leq i_2<\cdots<i_j \leq b} i_2\cdots i_j  \\
    & =
    \mathlarger{\sum}_{s=a}^{b-j+1}s\mathlarger{\sum}_{s+1 \leq i_2<\cdots<i_j\leq b} i_2\cdots i_j
\end{align*}
which gives a way to recursively evaluate the sum from the previous sum of the same form. However, when carried out in practice, this turns out to be increasingly complicated as $j$ gets larger. An interesting phenomenon that appears is that Faulhaber's formula is also involved in such calculations, which might be related to the fact that the formula is also an indispensable part in obtaining the matrix $M_d^{d+1}$.\\
On the other hand, when we try to find more coefficients from the matrix $M_d^{d+1}$, the main difficulty comes from two aspects: Firstly, we need information on coefficients $c(d,d),\cdots c(d,d-j+2)$ in order to find $c(d,d-j)$; Secondly, since there is no closed-form formula for the Bernoulli numbers, as they are involved in calculations, it is difficult to spot the pattern in which the coefficients evolve. We currently do not have a good way to overcome these difficulties and consequently propose the following open problem. 
\begin{openproblem}
Derive expressions for $c(d,d-j)$ from both the matrix $M_d^{d+1}$ and $(11)$ and prove that they are indeed equal.
\end{openproblem}


\begin{thebibliography}{9}
\bibitem{texbook}
Donald E. Knuth (1993). "Johann Faulhaber and sums of powers". Mathematics of Computation. 61 (203): 277–294. arXiv:math.CA/9207222
\bibitem{textbook}
Introductory Combinatorics (5th Edition), Richard A.Brualdi, Pearson
\bibitem{textbook}
Lectures on Contemporary Probability, Gregory F.Lawler, Lester N.Coyle, American Mathematical Society, Institute for Advanced Study
\bibitem{textbook}
Lectures on Contemporary Probability, Gregory F.Lawler, Lester N.Coyle, American Mathematical Society, Institute for Advanced Study
\bibitem{textbook}
Bourbaki, Nicolas (1987) [1981]. Topological Vector Spaces: Chapters 1–5. Éléments de mathématique. Translated by Eggleston, H.G.; Madan, S. Berlin New York: Springer-Verlag. ISBN 3-540-13627-4. OCLC 17499190.
\bibitem{textbook}
Khaleelulla, S. M. (1982). Counterexamples in Topological Vector Spaces. Lecture Notes in Mathematics. Vol. 936. Berlin, Heidelberg, New York: Springer-Verlag. ISBN 978-3-540-11565-6. OCLC 8588370.
\bibitem{textbook}
Narici, Lawrence; Beckenstein, Edward (2011). Topological Vector Spaces. Pure and applied mathematics (Second ed.). Boca Raton, FL: CRC Press. ISBN 978-1584888666. OCLC 144216834.
\bibitem{textbook}
Schaefer, Helmut H.; Wolff, Manfred P. (1999). Topological Vector Spaces. GTM. Vol. 8 (Second ed.). New York, NY: Springer New York Imprint Springer. ISBN 978-1-4612-7155-0. OCLC 840278135.
\bibitem{textbook}
Trèves, François (2006) [1967]. Topological Vector Spaces, Distributions and Kernels. Mineola, N.Y.: Dover Publications. ISBN 978-0-486-45352-1. OCLC 853623322.
\bibitem{textbook}
Wilansky, Albert (2013). Modern Methods in Topological Vector Spaces. Mineola, New York: Dover Publications, Inc. ISBN 978-0-486-49353-4. OCLC 849801114.
\end{thebibliography}
\end{document}